\documentclass[11pt,a4paper,leqno]{amsart}
\usepackage[T1]{fontenc}
\usepackage{lmodern,amsmath,amssymb,amsthm,mathtools,booktabs,microtype}
\usepackage[textwidth=160mm,textheight=232mm,centering]{geometry}
\usepackage[hidelinks]{hyperref}
\hypersetup{pdftitle={A characterization of the Ree curve},pdfauthor={Saeed Tafazolian},pdfkeywords={Ree curve, Frobenius, Weierstrass semigroup}}
\numberwithin{equation}{section}
\newtheorem{theorem}{Theorem}[section]
\newtheorem{proposition}[theorem]{Proposition}
\newtheorem{lemma}[theorem]{Lemma}
\newtheorem{corollary}[theorem]{Corollary}
\theoremstyle{remark}
\newtheorem{remark}[theorem]{Remark}
\newcommand{\F}{\mathbb F}
\newcommand{\PP}{\mathbb P}
\newcommand{\Gal}{\operatorname{Gal}}
\newcommand{\Tr}{\operatorname{Tr}}
\newcommand{\Diff}{\operatorname{Diff}}
\newcommand{\sw}{\operatorname{sw}}
\newcommand{\cR}{\mathcal R}
\newcommand{\A}{\mathcal A}
\newcommand{\V}{\mathcal V}
\newcommand{\T}{\mathcal T}
\newcommand{\Span}{\operatorname{span}}
\title[A characterization of the Ree curve]
{A characterization of the Ree curve}
\author{Saeed Tafazolian}
\dedicatory{In memory of Fernando Torres}
\address{IMECC/UNICAMP, R. S\'ergio Buarque de Holanda, 651, Cidade Universit\'aria ``Zeferino Vaz'', 13083-859 Campinas, SP, Brazil}
\email{tafazolian@ime.unicamp.br}
\date{}
\begin{document}
\begin{abstract}
Let $q=3q_0^2$, where $q_0=3^s$ and $s\ge1$. We show that a smooth,
projective, geometrically irreducible curve over $\F_q$ of genus
$\frac32q_0(q-1)(q+q_0+1)$ with $q^3+1$ rational points is
$\F_q$-isomorphic to the Ree curve. The proof compares the ramification
and Frobenius traces of a compositum with the Ree function field. For
$q_0=3$, a separate argument uses Hasse derivatives and pole orders.
\end{abstract}
\maketitle

\section{Introduction}
The Hermitian, Suzuki and Ree curves are the classical one-dimensional
Deligne--Lusztig curves associated respectively to the groups of types
${}^2A_2$, ${}^2B_2$ and ${}^2G_2$; see \cite{DL,HKT}. A natural question
is whether their genera and numbers of rational points determine them
over the ground field. This is known for the Hermitian and Suzuki curves
by \cite{RS,FT}. For the Ree curve, the question was considered in
\cite{TT}, explicitly stated as open in the introduction of
\cite{Skabelund}, and discussed in \cite[p.~40]{GKM}. We answer it
affirmatively.

Let $q=3q_0^2$, where $q_0=3^s$ and $s\ge1$, and put $k=\F_q$.
The Ree curve $\cR_q$ is the nonsingular model of
\begin{equation}\label{eq:ree-model}
 y^q-y=x^{q_0}(x^q-x),\qquad
 z^q-z=x^{2q_0}(x^q-x).
\end{equation}
It has genus $\frac32q_0(q-1)(q+q_0+1)$ and $q^3+1$ rational points
over $k$; see \cite{HP,Ped,HKT}.
Hansen and Pedersen \cite{HP} characterized the Ree curve by its genus,
number of $k$-rational points, and automorphism group isomorphic to
$\operatorname{Ree}(q)$; this characterization is also recalled in
\cite[p.~40]{GKM}. The question is whether the automorphism-group
hypothesis can be omitted. Our main result shows that it can: the genus
and number of rational points alone suffice.

\begin{theorem}\label{thm:main}
Let $q=3q_0^2$, where $q_0=3^s$ and $s\ge1$, and let $X$ be a smooth,
projective, geometrically irreducible curve over $\F_q$. If
\begin{equation}\label{eq:hypotheses}
 \#X(\F_q)=q^3+1,\qquad
 g(X)=\frac32q_0(q-1)(q+q_0+1),
\end{equation}
then $X$ is $\F_q$-isomorphic to $\cR_q$.
\end{theorem}

Thus the Ree curve has the same uniqueness property as the Hermitian
and Suzuki curves. No hypothesis on the automorphism group or on the
Weierstrass semigroup is required. In \cite[Theorem~4.6]{TT}, the
conclusion was proved for $q_0\ge9$ under the additional assumption
that some rational point $P$ satisfies
\[
 q^2+q_0q,\quad q^2+2q_0q\in H(P).
\]
Here $H(P)$ is the Weierstrass semigroup at $P$. Theorem~\ref{thm:main}
removes these assumptions and includes the smallest parameter $q_0=3$.
Related descriptions of the Ree curve include its realization as a ray
class field \cite{Lauter} and its projective embeddings
\cite{ED,Skabelund}; see also the survey \cite{GKM}.

We describe the proof. The Frobenius polynomial and divisor relation
established in \cite{TT} yield a separating function $x$ with pole
divisor $q^2P$ at a rational point $P$. We compare $F=k(X)$ with the Ree
function field $R$ over $K=k(x)$. The compositum $FR/F$ is an elementary
abelian $3$-extension, totally ramified at $P$ and split at all other
rational points. A local $q_0$-power approximation to $x$ bounds the
conductors of this extension. Positive trigonometric polynomials applied
to its relative Frobenius traces give a lower bound for the genus. These
bounds are incompatible unless $F=R$ for $q_0\ge9$. For $q_0=3$, they
reduce a counterexample to $F\cap R=K$.

The remaining argument is given in Section~\ref{sec:q3}. A cubic
operator obtained from the Frobenius identity and the first two Hasse
derivatives controls functions of small pole order. Applied to a function
of second non-gap order, it gives a differential normal form. The third
Hasse derivative then produces a function in $F\cap R$ whose pole order
is not divisible by $q^2$, a contradiction. The local expansions and
pole estimates used in this step are included in the proof.

Throughout the paper, a curve is smooth, projective and geometrically
irreducible. We write $L(D)$ for the Riemann--Roch space of a divisor $D$
and $0=m_0(P)<m_1(P)<\cdots$ for the non-gaps at a rational point $P$.
All composita are taken in a fixed algebraic closure of $k(x)$.

\section{Preliminaries}\label{sec:prelim}
We recall the facts used in the proof. General references for function
fields and curves over finite fields are \cite{Stichtenoth,HKT}.

\subsection{The Ree extension and its cubic subfields}
Let $K=k(x)$ and let $R=K(y,z)$ be defined by \eqref{eq:ree-model}. Then
\begin{equation}\label{eq:ree-degree}
 [R:K]=q^2,\qquad \Gal(R/K)\cong(k^2,+),
\end{equation}
where $(\alpha,\eta)$ sends $(y,z)$ to $(y+\alpha,z+\eta)$.
The pole of $x$ is totally ramified, every finite place is unramified,
and each place $x=a$, $a\in k$, splits completely; see \cite{HP,Ped}.

For a field $M$ of characteristic three, put
$\wp(M)=\{a^3-a:a\in M\}$. We shall use
\begin{equation}\label{eq:wp-power}
 a^{3^j}\equiv a\pmod{\wp(M)}.
\end{equation}
Note that $q=3^{2s+1}$. By the trace pairing, the nonzero $\F_3$-linear
functionals on $\Gal(R/K)$ are
\[
 (\alpha,\eta)\longmapsto
 \Tr_{k/\F_3}(\beta\alpha+\gamma\eta),
 \qquad (\beta,\gamma)\ne(0,0).
\]
Indeed, for $w=\beta y+\gamma z$ the function
$\sum_{i=0}^{2s}w^{3^i}$ satisfies the cubic equation with right side
$\beta x^{q_0}(x^q-x)+\gamma x^{2q_0}(x^q-x)$.
By \eqref{eq:wp-power} the same class has representative
\begin{equation}\label{eq:cubic-polynomial}
 p_{\beta,\gamma}(x)=
 \beta^{1/q_0}x^{3q_0+1}-\beta x^{q_0+1}
 +\gamma^{1/q_0}x^{3q_0+2}-\gamma x^{2q_0+1}.
\end{equation}
Its leading degree is prime to three. Thus its Swan conductor at
infinity is
\begin{equation}\label{eq:ree-swan}
 \sw_\infty(p_{\beta,\gamma})=
 \begin{cases}3q_0+1,&\gamma=0,\\3q_0+2,&\gamma\ne0.
 \end{cases}
\end{equation}

\subsection{Ramification}
A local Artin--Schreier class represented by a function of pole order at
most $b$ has Swan conductor at most $b$; see
\cite[Proposition~3.7.8]{Stichtenoth}. For a ramified one-dimensional
character its Artin conductor is one more than its Swan conductor.
If $L/M$ is a finite abelian extension, the conductor--discriminant formula is
\begin{equation}\label{eq:conductor-discriminant}
 v_P(\mathfrak d_{L/M})=\sum_{\chi\ne1}\mathfrak f_P(\chi);
\end{equation}
see \cite[Ch.~VI, \S\S2--3]{Serre}.
In an abelian extension the inertia groups above a place agree. The fixed
field of this inertia group is therefore unramified above that place.

We also use the unramified case of the Deuring--Shafarevich formula. If
$L/M$ is geometric, unramified and cyclic of degree $p$, then
\[
 \gamma_p(L)-1=p(\gamma_p(M)-1),
\]
where $\gamma_p$ is the $p$-rank; see \cite[Theorem~11.62]{HKT}.
In particular a function field of $p$-rank zero has no nontrivial
geometric unramified Galois $p$-extension.

\begin{lemma}\label{lem:local-pullback}
Let $K_0=k((t))\subset K_1=k((u))$ be a totally ramified separable
extension of characteristic three, with $k$ perfect; both fields are complete and have residue field $k$. Put
\[
 e=v_{K_1}(t),\qquad d_0=v_{K_1}(dt/du),
\]
and suppose that $3\mid e$. A cubic character of $K_0$ of Swan conductor
at most $c$ restricts to a character of $K_1$ of Swan conductor at most
\begin{equation}\label{eq:local-pullback}
 \max\{ec/3,\ e(c+1)-d_0-1\}.
\end{equation}
\end{lemma}
\begin{proof}
The different exponent is $d_0$; see \cite[Ch.~III, \S6]{Serre}.
Separate the terms of the Laurent expansion of $t$ whose exponents are
divisible by three. Since $k$ is perfect, this gives
\[
 t=A^3+\eta,\qquad v(A)=e/3,\qquad v(\eta)=d_0+1>e.
\]
The last assertions follow because the leading exponent of $t$ is
divisible by three, whereas the first term that survives differentiation
has exponent $d_0+1$, which is not divisible by three and hence cannot equal $e$.

A standard Artin--Schreier representative has principal part
$\sum_{1\le j\le c}b_jt^{-j}$. For each term,
\[
 v(t^{-j}-A^{-3j})\ge d_0+1-(j+1)e,
 \qquad b_jA^{-3j}\equiv b_j^{1/3}A^{-j}\pmod{\wp(K_1)}.
\]
The latter term has pole order $je/3$. Taking the maximum gives
\eqref{eq:local-pullback}; regular terms and further reduction do not
increase the Swan conductor.
\end{proof}

\subsection{Hasse derivatives and linear series}
Let $F/k$ be a function field of characteristic three and let $x$ be
separating. Write $D_x^{(i)}$ for its Hasse derivatives and $D_x=D_x^{(1)}$.
We use the product rule and the composition rule
\begin{equation}\label{eq:hasse-composition}
 D_x^{(i)}D_x^{(j)}=\binom{i+j}{i}D_x^{(i+j)}.
\end{equation}
If $e$ is a power of three, then $[F:F^e]=e$ and
$1,x,\ldots,x^{e-1}$ is an $F^e$-basis of $F$. Indeed, $x$ is
separating, so $[F^e(x):F^e]=e=[F:F^e]$. Consequently,
\begin{equation}\label{eq:root-criterion}
 D_x^{(i)}a=0\ (1\le i<e)\quad\Longrightarrow\quad a\in F^e.
\end{equation}
To see this, write $a=\sum_{j=0}^{e-1}b_j^e x^j$. If $j>0$ is the
largest index with $b_j\ne0$, then $D_x^{(j)}a=b_j^e\ne0$.
This is the Hasse--Schmidt root criterion; see \cite{HasseSchmidt}.
In particular
\begin{equation}\label{eq:three-basis}
 F=F^3\oplus xF^3\oplus x^2F^3,\qquad \ker D_x=F^3.
\end{equation}
Thus, if $a=b^3+xc^3+x^2r^3$, then
$D_xa=c^3-xr^3$ and $D_x^{(2)}a=r^3$.
When $x,a,D_xa,D_x^{(2)}a$ are regular away from a point $P$, so are
$b,c,r$: successively use
\[
 r^3=D_x^{(2)}a,\qquad c^3=D_xa+xr^3,\qquad
 b^3=a-xc^3-x^2r^3
\]
and $v_Q(f^3)=3v_Q(f)$ at each $Q\ne P$.

For a separating local parameter $z$, the degree-three change-of-variable
formula is
\begin{equation}\label{eq:chain-three}
 D_z^{(3)}a=(D_xa)D_z^{(3)}x+
 2(D_x^{(2)}a)(D_zx)(D_z^{(2)}x)+(D_x^{(3)}a)(D_zx)^3.
\end{equation}
Indeed, write $x(z+T)-x(z)=x_1T+x_2T^2+x_3T^3+O(T^4)$, where
$x_i=D_z^{(i)}x$. In the expansion
\[
 a(z+T)=\sum_{j\ge0}D_x^{(j)}a\,
                 \bigl(x(z+T)-x(z)\bigr)^j,
\]
the coefficient of $T^3$ is
$(D_xa)x_3+2(D_x^{(2)}a)x_1x_2+(D_x^{(3)}a)x_1^3$.
This proves \eqref{eq:chain-three}; the expansion is the Taylor
formula for Hasse derivatives \cite{HasseSchmidt}.

We recall two consequences of the St\"ohr--Voloch theory \cite{SV}.
For a base-point-free plane series of degree $d$ and nonzero
$\F_q$-Frobenius order $\nu$, its Frobenius divisor has degree
$\nu(2g-2)+(q+2)d$ and multiplicity at least two at every rational point.
Hence
\begin{equation}\label{eq:SV-bound}
 \nu(2g-2)+(q+2)d\ge2\#X(k).
\end{equation}
The $p$-adic criterion says that if $\epsilon$ is a generic order and
$\binom{\epsilon}{j}\not\equiv0\pmod p$, then $j$ is a generic order.
For the series $\langle1,x,a\rangle$, its generic orders are
$(0,1,\epsilon)$, where $\epsilon$ is the first $j>1$ with
$D_x^{(j)}a\ne0$. We shall also use the inequality
$j_i(P)\ge\epsilon_i$ between rational and generic orders.

\subsection{Relative Frobenius traces}
For a curve $X'/k$, the trace formula gives
\[
 \#X'(\F_{q^i})=q^i+1-
 \Tr(\operatorname{Frob}_q^i\mid H^1_{\mathrm{\acute et}}(X'_{\bar k},\mathbf Q_\ell)),
 \qquad \ell\ne3.
\]
If $\pi:Y\to X$ is finite separable of degree $d$, then
$\pi_*\pi^*=d$ on $H^1$. The projector $d^{-1}\pi^*\pi_*$ therefore
splits off $H^1(X)$, with Frobenius-stable complement of dimension
$2(g(Y)-g(X))$; see \cite[Ch.~VI, \S\S11--12]{Milne}.

\begin{lemma}\label{lem:positive-trace}
Let $\pi:Y\to X$ be as above, put $\Delta=g(Y)-g(X)$, and write
$N_i=\#X(\F_{q^i})$, $M_i=\#Y(\F_{q^i})$. For real numbers
$c_0,\ldots,c_r$, not all zero, put
\[
 a_i=\sum_{j=0}^{r-i}c_jc_{j+i}\quad(0\le i\le r).
\]
If $a_0>0$, then
\begin{equation}\label{eq:positive-trace}
 \Delta\ge\frac1{a_0}\sum_{i=1}^r a_i\frac{M_i-N_i}{q^{i/2}}.
\end{equation}
If $a_i\ge0$ for $i\ge1$, each $M_i$ on the right may be replaced by
a lower bound for $M_i$.
\end{lemma}
\begin{proof}
Let $\alpha_1,\ldots,\alpha_{2\Delta}$ be the Frobenius eigenvalues
on the complement, viewed in $\mathbf C$, and put $z_j=\alpha_j/\sqrt q$. Then $|z_j|=1$ and
$\sum_j\alpha_j^i=N_i-M_i$. With $Q(T)=\sum c_jT^j$, we have
\[
 0\le\sum_{j=1}^{2\Delta}|Q(z_j)|^2
 =2\Delta a_0+2\sum_{i=1}^r a_i\frac{N_i-M_i}{q^{i/2}},
\]
which proves the assertion.
\end{proof}

\section{The Frobenius linear series}\label{sec:frobenius}
From now on let $X/k$ satisfy \eqref{eq:hypotheses}. Put $F=k(X)$,
$g=g(X)$, and
\begin{equation}\label{eq:constants}
 N=q^2,\qquad m=(q+1)(q+3q_0+1),\qquad C=2g-1.
\end{equation}
We shall use the identities
\begin{align}
 2g&=(3q_0+1)N-q-3q_0,\label{eq:genus-identity}\\
 2g-2&=(3q_0-2)m,\qquad q^3+1=(q-3q_0+1)m,
 \label{eq:SV-identities}\\
 q(m-N)-m&=C.\label{eq:C-identity}
\end{align}

\begin{proposition}\label{prop:frobenius-data}
The characteristic polynomial of Frobenius on $\operatorname{Jac}(X)$ is
\begin{equation}\label{eq:frobenius-polynomial}
 (T^2+q)^{A_0}(T^2+3q_0T+q)^{B_0},
\end{equation}
where
\[
 A_0=\frac{q_0}{2}(q-1)(q+3q_0+1),\qquad B_0=q_0(q^2-1).
\]
In particular,
\begin{equation}\label{eq:first-five}
 \#X(\F_{q^i})=q^3+1\quad(1\le i\le5),
\end{equation}
and $X$ has $3$-rank zero. If $P\in X(k)$ and $\Phi$ denotes the
$q$-Frobenius, then
\begin{equation}\label{eq:frobenius-divisor}
 NQ+3q_0q\Phi(Q)+2q\Phi^2(Q)+3q_0\Phi^3(Q)+\Phi^4(Q)\sim mP
\end{equation}
for every geometric point $Q$.
\end{proposition}
\begin{proof}
We give the explicit-formula argument of \cite[Proposition~3.1]{TT}
with the numerical identity written out. Let $\sqrt q\,e^{\pm\sqrt{-1}\,\theta_j}$,
$1\le j\le g$, be the Frobenius eigenvalues, and put
\[
 f(\theta)=\frac13(1+\sqrt3\cos\theta+\cos2\theta)^2
 =1+2\sum_{i=1}^4c_i\cos i\theta,
 \qquad (c_1,c_2,c_3,c_4)
 =\left(\frac{\sqrt3}2,\frac7{12},\frac{\sqrt3}6,\frac1{12}\right).
\]
Write $N_i=\#X(\F_{q^i})$ and $\psi(t)=\sum_{i=1}^4c_it^i$.
The point-count formula and $N_i\ge N_1=q^3+1$ give
\[
 0\le\sum_{j=1}^g f(\theta_j)
 =g+\psi(\sqrt q)+\psi(q^{-1/2})
       -\sum_{i=1}^4c_iN_iq^{-i/2}
 \le g+\psi(\sqrt q)-q^3\psi(q^{-1/2})=0.
\]
For the last equality, substitute $q=3q_0^2$ and use
\[
 \psi(\sqrt q)=\frac32q_0+\frac74q_0^2+\frac32q_0^3+\frac34q_0^4,
 \qquad
 q^3\psi(q^{-1/2})
 =\frac{27}2q_0^5+\frac{21}4q_0^4+\frac32q_0^3+\frac14q_0^2.
\]
Thus each $f(\theta_j)$ vanishes, so
$\cos\theta_j\in\{0,-\sqrt3/2\}$. The eigenvalues are consequently
roots of $T^2+q$ or $T^2+3q_0T+q$. If their multiplicities are $A_0,B_0$,
then $A_0+B_0=g$ and $3q_0B_0=N_1-q-1=q(q^2-1)$.
These give the stated exponents and \eqref{eq:frobenius-polynomial}.

Frobenius is semisimple, so its minimal polynomial is
$(T^2+q)(T^2+3q_0T+q)$; both factors occur and are distinct.
Apply this polynomial to $[Q-P]\in\operatorname{Jac}(X)$.
Since $\Phi(P)=P$ and the value of the polynomial at $1$ is $m$,
its expansion gives \eqref{eq:frobenius-divisor}; compare
\cite[(3.4)]{TT}.

For the point counts, let $S_i(b)$ be the power sum of the roots of
$T^2+bT+q$. The recurrence $S_i(b)=-bS_{i-1}(b)-qS_{i-2}(b)$,
with $S_0(b)=2$ and $S_1(b)=-b$, gives
\[
\begin{array}{c|rrrrr}
 i&1&2&3&4&5\\\hline
 S_i(0)&0&-2q&0&2q^2&0\\
 S_i(3q_0)&-3q_0&q&0&-q^2&3q_0q^2.
\end{array}
\]
Using $3q_0B_0=q(q^2-1)$ and $2A_0-B_0=q(q-1)$ in
$\#X(\F_{q^i})=q^i+1-A_0S_i(0)-B_0S_i(3q_0)$ proves
\eqref{eq:first-five}, including the fifth count. The reciprocal polynomial
$(1+qT^2)^{A_0}(1+3q_0T+qT^2)^{B_0}$ is $1$ modulo three, so the
$3$-rank is zero by \cite{Sti79}.
\end{proof}

\begin{proposition}\label{prop:first-nongap}
At every $P\in X(k)$ one has $m_1(P)=N$. There are $x,f\in F$ with
\begin{equation}\label{eq:xf-poles}
 (x)_\infty=NP,\qquad (f)_\infty=mP,
\end{equation}
and $x$ is separating, with $[F:k(x)]=N$.
\end{proposition}
\begin{proof}
We recall the argument for the first assertion; cf.
\cite[Proposition~3.4]{TT}. If $Q\notin X(k)$, the six points
$Q,\Phi(Q),\ldots,\Phi^5(Q)$ are distinct by \eqref{eq:first-five}.
Subtracting \eqref{eq:frobenius-divisor} at $Q$ from the same relation
at $\Phi(Q)$ gives a principal divisor whose only negative coefficient
is $-N$ at $Q$: the other coefficients are
\[
 N-3q_0q,\quad3q_0q-2q,\quad2q-3q_0,\quad3q_0-1,\quad1,
\]
all positive. Hence $N\in H(Q)$.

To pass to rational points, let $p_2:\bar X\times\bar X\to\bar X$
be the second projection and let $\Delta_{\bar X}$ be the diagonal.
The line bundle $\mathcal O(N\Delta_{\bar X})$ is flat over the second
factor, and its restriction to the fibre over $Q$ is
$\mathcal O_{\bar X}(NQ)$. Since $p_2$ is proper, upper semicontinuity
\cite[III, Theorem~12.8]{Hartshorne} shows that
\[
 \{Q\in\bar X:h^0(\bar X,\mathcal O_{\bar X}(NQ))\ge2\}
\]
is closed. It contains the complement of the finite set $X(k)$ by
the preceding argument, and hence equals $\bar X$. For rational $P$,
extension of scalars preserves global sections, so $\ell_k(NP)\ge2$.
Thus $m_1(P)\le N$. Conversely, a function of pole order $m_1(P)$
defines a map to $\PP^1$ of degree $m_1(P)$, with only $P$ above
infinity. Each of its $q$ finite rational fibres has at most
$m_1(P)$ rational points. This is Lewittes' bound \cite{Lewittes}:
\[
 q^3+1\le q\,m_1(P)+1.
\]
It gives $m_1(P)\ge N$, proving the equality.

Choose $x$ with the stated pole divisor. If $x$ were not separating, it
would be a cube in $F$, producing the non-gap $N/3$, a contradiction.
Thus $x$ is separating and $[F:k(x)]=N$.
Finally, for a rational point $Q\ne P$, \eqref{eq:frobenius-divisor}
gives $mQ\sim mP$. A function with divisor $mQ-mP$ gives $f$.
\end{proof}

\begin{lemma}\label{lem:plane-frobenius-order}
Let $u\in L(mP)\setminus\Span_k\{1,x\}$ and let $\nu$ be the nonzero
$\F_q$-Frobenius order of the plane series $\langle1,x,u\rangle$.
Then
\begin{equation}\label{eq:order-bound}
 \nu\ge\frac{q-6q_0}{3q_0-2}
 =q_0-2+\frac{2q_0-4}{3q_0-2}>q_0-2,
 \qquad \nu\ge q_0-1\ge2.
\end{equation}
\end{lemma}
\begin{proof}
After removing the fixed divisor, the series has degree $d_u\le m$.
By \eqref{eq:SV-bound} and \eqref{eq:SV-identities},
\begin{align*}
 \nu
 &\ge\frac{2(q^3+1)-(q+2)d_u}{2g-2}\\
 &\ge\frac{2(q^3+1)-(q+2)m}{2g-2}
 =\frac{q-6q_0}{3q_0-2}.
\end{align*}
The last equality in \eqref{eq:order-bound} uses $q=3q_0^2$.
Since $q_0\ge3$ and $\nu$ is an integer, the conclusion follows.
\end{proof}

\begin{proposition}\label{prop:section-identities}
Fix $P\in X(k)$ and choose $x,f$ as above. Every $u\in L(mP)$ satisfies
\begin{equation}\label{eq:section-frobenius}
 u^q-u=(x^q-x)D_xu,\qquad D_x^{(2)}u=0.
\end{equation}
Moreover
\begin{equation}\label{eq:dx}
 (dx)=(2g-2)P.
\end{equation}
If $u\notin\Span_k\{1,x\}$ has pole order $t$, then
\begin{equation}\label{eq:derivative-root}
 D_xu=h_u^{q_0},\qquad
 (h_u)_\infty=\frac{q(t-N)}{q_0}P
\end{equation}
for some $h_u\in F$. The generic orders of $\langle1,x,u\rangle$ are
$(0,1,\epsilon)$ with $\epsilon\in\{q_0,3q_0\}$.
\end{proposition}
\begin{proof}
The first identity is immediate for $u\in\Span_k\{1,x\}$.
Otherwise, Lemma~\ref{lem:plane-frobenius-order} gives $\nu>1$.
The Frobenius determinant of orders $0,1$ therefore vanishes, giving
\eqref{eq:section-frobenius}. Differentiation yields
$(x^q-x)D_xD_xu=0$. By \eqref{eq:hasse-composition},
$D_xD_xu=2D_x^{(2)}u$, so $D_x^{(2)}u=0$.

Apply the identity to $f$. Since $3\nmid m$, one has
$v_P(df)=-m-1$, whereas $-v_P(D_xf)=q(m-N)$.
Thus $v_P(dx)=q(m-N)-m-1=2g-2$. Since the degree of a canonical divisor
is $2g-2$ and $x$ is regular off $P$, this proves \eqref{eq:dx}.

For $j>1$, the Frobenius determinant of orders $0,j$ is
$-(x^q-x)D_x^{(j)}u$. Consequently $\nu=\epsilon$, the third generic
order. The second identity in \eqref{eq:section-frobenius} excludes
$\epsilon=2$. We give the $3$-adic step explicitly. Write
$\epsilon=\sum d_i3^i$. By Lucas' theorem, the number of indices $j$
for which $\binom\epsilon j\not\equiv0\pmod3$ is $\prod_i(d_i+1)$.
Every such $j$ must be one of the three generic orders
$0,1,\epsilon$. If there were three such indices, we would have
$\epsilon=2\cdot3^a$, and the intermediate index $3^a$ would have
to equal $1$. This would give $\epsilon=2$, already excluded.
There are therefore exactly two such indices, so precisely one digit
of $\epsilon$ is $1$ and the others are zero. Hence $\epsilon$ is a
power of three. In conjunction with \eqref{eq:order-bound}, this gives
$\epsilon\ge q_0$.

For $1\le j<\epsilon$, the composition rule gives
\[
 D_x^{(j)}D_xu=(j+1)D_x^{(j+1)}u=0.
\]
For $j<\epsilon-1$ this follows from the generic orders, and for
$j=\epsilon-1$ from $3\mid\epsilon$. Hence $D_xu=a^\epsilon$ for
some $a\in F$. Since $u\notin\Span_k\{1,x\}$ and $m_1(P)=N$, we have
$t>N$, so the Frobenius identity shows that $a$ is nonconstant. Equation \eqref{eq:dx} shows that $a$ is regular away
from $P$, and the Frobenius identity gives
\[
 -v_P(a)=\frac{q(t-N)}{\epsilon}\ge N.
\]
Therefore
\[
 q_0\le\epsilon\le\frac{t-N}{q}\le\frac{m-N}{q}<3q_0+3.
\]
As $\epsilon$ is a power of three, it is $q_0$ or $3q_0$.
Taking $h_u=a$ or $a^3$, respectively, proves
\eqref{eq:derivative-root}.
\end{proof}

\begin{corollary}\label{cor:finite-splitting}
The extension $F/k(x)$ is unramified away from $P$, and every finite
rational place of $k(x)$ splits completely.
\end{corollary}
\begin{proof}
The ramification assertion follows from \eqref{eq:dx}. Each of the
$q^3$ rational points different from $P$ is a simple zero of $x^q-x$.
Its pole order is $qN=q^3$, so these are all its zeros. Every fibre
$x=\alpha$, $\alpha\in k$, therefore consists of $N$ rational points.
\end{proof}

\begin{lemma}\label{lem:second-nongap-upper}
For every $P\in X(k)$, one has $m_2(P)\le m-2q$.
\end{lemma}
\begin{proof}
Put $\mathcal D=|mP|$ and $r=\dim\mathcal D$. This series is
base-point-free: the section $1$ does not vanish away from $P$,
and a function with pole divisor $mP$ gives a section not vanishing
at $P$. For a nonrational
geometric point $Q$, the points in \eqref{eq:frobenius-divisor} are
distinct. Applying that relation at $Q$, $\Phi^{-1}(Q)$ and
$\Phi^{-2}(Q)$ gives effective divisors in $\mathcal D$ with respective
multiplicities $N$, $3q_0q$ and $2q$ at $Q$. Thus these three distinct
integers are generic orders of $\mathcal D$; in particular,
$\epsilon_{r-2}\ge2q$.

The orders of $|mP|$ at $P$ are $m-m_r(P),\ldots,m-m_0(P)$, since
the vanishing order at $P$ of the section represented by a function
of pole order $m_i(P)$ is $m-m_i(P)$. Therefore
\[
 m-m_2(P)=j_{r-2}(P)\ge\epsilon_{r-2}\ge2q.
\]
This is the part of \cite[Proposition~3.10]{TT} needed here, now derived
directly from the divisor relation.
\end{proof}

\begin{remark}\label{rem:TT-input}
Proposition~\ref{prop:frobenius-data} includes a proof of the Frobenius
polynomial and divisor relation used in \cite{TT}. The displayed
inequality $(q-6q_0)/(3q_0-2)\ge q_0-1$ in
\cite[Proposition~3.7]{TT} is incorrect: the difference is
$-(q_0+2)/(3q_0-2)$. Lemma~\ref{lem:plane-frobenius-order}
obtains the required bound by integrality of the Frobenius order.
The exclusion of plane order two in \cite[Lemma~4.4]{TT} required
$\nu>2$; Proposition~\ref{prop:section-identities} instead proves
$D_x^{(2)}u=0$, which also excludes that order when $q_0=3$.
Neither of the additional non-gap assumptions of
\cite[Theorem~4.6]{TT} is used.
\end{remark}

\section{The compositum and its genus}\label{sec:compositum}
Fix $P\in X(k)$ and the separating function $x$ above. Put
\[
 K=k(x),\qquad E=FR,\qquad J=F\cap R,\qquad d=[E:F].
\]
Let $\mathcal Y/k$ be the curve with function field $E$. Since $R/K$ is Galois,
restriction gives
\begin{equation}\label{eq:compositum-group}
 H:=\Gal(E/F)\cong\Gal(R/J)\le(k^2,+).
\end{equation}
In particular $d$ is a power of three.

\begin{proposition}\label{prop:compositum-splitting}
The full constant field of $E$ is $k$. The extension $E/F$ is totally
ramified at $P$, unramified elsewhere, and every point of
$X(k)\setminus\{P\}$ splits completely. Thus
\begin{equation}\label{eq:compositum-counts}
 \#\mathcal Y(\F_{q^i})=dq^3+1\qquad(1\le i\le5).
\end{equation}
\end{proposition}
\begin{proof}
The defining equations of $R/K$ are integral at every finite place and
have derivative $-1$. They define an \emph{\'etale} algebra there;
this property is preserved by base change. Thus $E/F$ is unramified off
$P$. At a rational point different from $P$, the equations reduce to
$y^q-y=z^q-z=0$, so that point splits completely. In particular $E$
has a rational place and its full constant field is $k$.

Let $I\le H$ be the inertia group above $P$. Since $H$ is abelian,
$E^I/F$ is unramified everywhere and geometric. If $I\ne H$, it has
a cyclic subextension of degree three, contrary to the $3$-rank zero
of $F$ and the Deuring--Shafarevich formula. Hence $I=H$.
Finally \eqref{eq:first-five} implies $X(\F_{q^i})=X(k)$ for
$1\le i\le5$, proving \eqref{eq:compositum-counts}.
\end{proof}

\begin{lemma}\label{lem:approximation}
There are $A,\delta\in F$ with
\begin{equation}\label{eq:approximation}
 x=A^{q_0}+\delta,\qquad v_P(A)=-N/q_0,\qquad v_P(\delta)=C.
\end{equation}
\end{lemma}
\begin{proof}
Write $D_xf=u^{q_0}$ as in \eqref{eq:derivative-root} and put
$v=f^{3q_0}-x^{3q_0}u$. The Frobenius identity gives
$v^{q_0}=f-xu^{q_0}$, whence
\[
 x=(-v/u)^{q_0}+f/u^{q_0}.
\]
Take $A=-v/u$ and $\delta=f/u^{q_0}$. Then
$v_P(\delta)=-m+q(m-N)=C>0$. Since $v_P(x)=-N$, the valuation of
$A$ follows.
\end{proof}

\begin{proposition}\label{prop:conductor}
Every nontrivial character of $H$ has Swan conductor at $P$ at most
\begin{equation}\label{eq:b}
 b=(3q_0+2)N/q_0=3N+6q_0q.
\end{equation}
Consequently, writing $\Delta=g(E)-g$, if $d>1$ then
\begin{equation}\label{eq:upper}
 \frac{\Delta}{(d-1)N}
 \le\frac{3q_0}{2}+2+\frac1{q_0}-\frac{q+3q_0+1}{2N}.
\end{equation}
\end{proposition}
\begin{proof}
Fix a nontrivial additive character of $\F_3$. Composing with it
identifies nontrivial characters of $H$ with nonzero linear functionals
$\lambda:H\to\F_3$. Extend $\lambda$ to a linear functional
$\widetilde\lambda:(k^2,+)\to\F_3$. The trace pairing represents it by
a pair $(\beta,\gamma)$ in \eqref{eq:cubic-polynomial}.
Let $R'=R^{\ker\widetilde\lambda}$ be the associated cubic field.
Since $H\cap\ker\widetilde\lambda=\ker\lambda$, one has
$FR'\subseteq E^{\ker\lambda}$. Moreover $R'\not\subseteq F$:
otherwise $R'\subseteq F\cap R$ and
$H=\Gal(R/F\cap R)\subseteq\ker\widetilde\lambda$, contrary to
$\lambda\ne0$. Therefore $F\cap R'=K$ and $[FR':F]=3$.
Both fields in the preceding inclusion have degree three over $F$, so
\begin{equation}\label{eq:character-base-change}
 FR'=E^{\ker\lambda}.
\end{equation}
Thus the pullback of \eqref{eq:cubic-polynomial} represents the
character under consideration.

Write $p(T)=p_{\beta,\gamma}(T)=\sum c_jT^j$ and
$p^{[-q_0]}(T)=\sum c_j^{1/q_0}T^j$. By
\eqref{eq:wp-power} and \eqref{eq:approximation}, its local class over
$F$ has representative
\[
 p^{[-q_0]}(A)+p(x)-p(A^{q_0}).
\]
The first term has pole order at most $(3q_0+2)N/q_0=b$. For the second,
\[
 v_P(x^j-A^{q_0j})\ge C-(j-1)N,
\]
so its pole order is at most
$(3q_0+1)N-C=q+3q_0+1<b$. This proves the conductor bound.

All nontrivial characters are ramified at $P$, by
Proposition~\ref{prop:compositum-splitting}. The conductor--discriminant
formula gives
\[
 \deg\Diff(E/F)\le(d-1)(b+1).
\]
Riemann--Hurwitz therefore yields
\begin{equation}\label{eq:genus-upper-b}
 \Delta\le(d-1)\left(g+\frac{b-1}{2}\right).
\end{equation}
Substitution of \eqref{eq:genus-identity} and \eqref{eq:b} proves
\eqref{eq:upper}.
\end{proof}

We now apply Lemma~\ref{lem:positive-trace}. Put $N_i=\#X(\F_{q^i})$.
For $d>1$ and nonnegative autocorrelations $a_i$, it gives
\begin{equation}\label{eq:relative-bound}
 \frac{\Delta}{(d-1)N}\ge
 L_{\boldsymbol c}(q,d):=
 \frac1{a_0}\sum_{i=1}^r
 a_i\frac{dq^3+1-N_i}{(d-1)Nq^{i/2}},
\end{equation}
since $\#\mathcal Y(\F_{q^i})\ge\#\mathcal Y(k)=dq^3+1$. This expression is
nondecreasing in $d$, because
\[
 \frac{dq^3+1-N_i}{d-1}
 =q^3-\frac{N_i-(q^3+1)}{d-1},\qquad N_i\ge q^3+1.
\]
The two coefficient vectors used below have autocorrelations
\begin{equation}\label{eq:vectors}
\begin{array}{c|l}
 \boldsymbol c&(a_0,a_1,\ldots,a_r)\\ \hline
 (2,4,5,5,4,2)&(90,81,60,36,16,4)\\
 (4,8,11,13,14,13,11,8,4)&(936,890,773,614,441,280,152,64,16).
\end{array}
\end{equation}
We denote these vectors by $\boldsymbol c^{(5)}$ and
$\boldsymbol c^{(8)}$, respectively.

\begin{proposition}\label{prop:lower-five}
If $d>1$, then
\begin{equation}\label{eq:lower-five}
 \frac{\Delta}{(d-1)N}\ge
 \frac9{10}\sqrt q+\frac23+\frac{2}{5\sqrt q}
 +\frac8{45q}+\frac2{45q^{3/2}}.
\end{equation}
In particular Theorem~\ref{thm:main} holds for $q_0\ge27$.
\end{proposition}
\begin{proof}
Use $\boldsymbol c^{(5)}$ in \eqref{eq:relative-bound} and
$N_i=q^3+1$ for $1\le i\le5$. This gives
\eqref{eq:lower-five}.
If $q_0\ge27$ and $d>1$, the lower bound is strictly greater than
$\frac9{10}\sqrt3\,q_0+\frac23$, whereas \eqref{eq:upper} is less
than $\frac32q_0+2+1/q_0$. But
\[
 \frac9{10}\sqrt3>\frac{14}{9},\qquad
 \frac{14q_0}{9}+\frac23-\left(\frac32q_0+2+\frac1{q_0}\right)
 =\frac{q_0}{18}-\frac43-\frac1{q_0}\ge\frac7{54}>0.
\]
This is impossible. Hence $d=1$, so $R\subseteq F$. Since both
extensions have degree $N$ over $K$, they are equal.
\end{proof}

\section{The small parameters: genus bounds}\label{sec:small-parameters}
The preceding section proves the theorem for $q_0\ge27$. Thus only
$q_0=9$ and $q_0=3$ remain. In this section the genus bounds settle
$q_0=9$ and reduce a counterexample for $q_0=3$ to $F\cap R=k(x)$.
The latter case is completed in Section~\ref{sec:q3}.

We first use the intersection $J=F\cap R$ to improve the conductor
bound. Put
\[
 e_J=[J:K],\qquad n=[F:J]=N/e_J.
\]
Thus $d=n$. The character group of $J/K$ is an $\F_3$-subspace
$\Lambda\subset k^2$ of order $e_J$, through
\eqref{eq:cubic-polynomial}. Set
\[
 t_J=|\Lambda\cap(k\times\{0\})|.
\]

\begin{lemma}\label{lem:intersection-bound}
If $F\ne R$, every nontrivial character of $E/F$ has Swan conductor
at $P$ at most
\begin{equation}\label{eq:intersection-b}
 b_J=\max\left\{\frac{N(t_J+3q_0+1)}{3e_J},\ q+3q_0+1\right\}.
\end{equation}
\end{lemma}
\begin{proof}
There are $t_J-1$ nontrivial characters of $J/K$ with $\gamma=0$,
and $e_J-t_J$ with $\gamma\ne0$. By \eqref{eq:ree-swan}, the
conductor--discriminant formula and Riemann--Hurwitz give
\begin{equation}\label{eq:genus-J}
 2g(J)=e_J(3q_0+1)-t_J-3q_0.
\end{equation}
A cubic subextension $J'/J$ of $R/J$ corresponds to a character space
$\Lambda'\supset\Lambda$ of order $3e_J$. Its intersection with
$k\times\{0\}$ has order $t_J$ or $3t_J$. Applying
\eqref{eq:genus-J} to $J'$ and using Riemann--Hurwitz for $J'/J$ shows
that its Swan conductor is, respectively,
\[
 t_J+3q_0+1\quad\hbox{or}\quad3q_0+1.
\]

Both $F/K$ and $J/K$ are totally ramified at infinity and unramified
elsewhere. The different exponent of $F/J$ is therefore
\begin{align}
 d_0&=(2g-2)-n(2g(J)-2)\notag\\
 &=n(t_J+3q_0+2)-(q+3q_0+2).\label{eq:different-FJ}
\end{align}
As $F\ne R$, we have $n>1$ and $3\mid n$. Also $F\cap J'=J$, so
$FJ'/F$ is a cubic quotient of $E/F$. Apply
Lemma~\ref{lem:local-pullback} with $e=n$ and
$c=t_J+3q_0+1$. Its second bound is $q+3q_0+1$, giving
\eqref{eq:intersection-b}.
\end{proof}

\begin{lemma}\label{lem:small-intersection}
If $q_0\in\{3,9\}$ and $F\ne R$, then $e_J\le3$.
\end{lemma}
\begin{proof}
Suppose that $e_J\ge9$. Since $t_J\le e_J$,
\eqref{eq:intersection-b} gives $b_J\le513$ for $q_0=3$ and
$b_J\le80919$ for $q_0=9$. By \eqref{eq:genus-upper-b}, these give
upper bounds $3883/729$ and $867010/59049$ for
$\Delta/((d-1)N)$, respectively. The lower bounds in
\eqref{eq:lower-five} are
\[
 \frac{818}{1215}+\frac{12005}{4374}\sqrt3\quad(q_0=3),
 \qquad
 \frac{7298}{10935}+\frac{4791721}{590490}\sqrt3\quad(q_0=9).
\]
Using $\sqrt3>433/250$, their differences from the respective upper
bounds are greater than
\[
 \frac{21973}{218700}>0,\qquad
 \frac{1937731}{49207500}>0.
\]
This contradiction proves the assertion.
\end{proof}

The higher point counts needed below are determined by
\eqref{eq:frobenius-polynomial}. For $S_i(b)$ the power sum of the
roots of $T^2+bT+q$, one has
\begin{equation}\label{eq:point-recurrence}
 S_0(b)=2,\quad S_1(b)=-b,\quad
 S_i(b)=-bS_{i-1}(b)-qS_{i-2}(b),\quad
 N_i=q^i+1-A_0S_i(0)-B_0S_i(3q_0).
\end{equation}
Thus
\begin{equation}\label{eq:small-counts}
\begin{array}{c|r|r}
 i&q=27&q=243\\\hline
 1,2,3,4,5&19684&14348908\\
 6&530200972&229611338954164\\
 7&10073464156&49825657453689460\\
 8&282056464900&12157460415213443596.
\end{array}
\end{equation}

In the following comparisons we use $\sqrt3>433/250$, since
$3\cdot250^2-433^2=11>0$.

\subsection{Completion for \texorpdfstring{$q_0=9$}{q0=9}}
\begin{proposition}\label{prop:q9}
Theorem~\ref{thm:main} holds for $q_0=9$.
\end{proposition}
\begin{proof}
Suppose $F\ne R$. By Lemma~\ref{lem:small-intersection},
$e_J\in\{1,3\}$, so $d=N/e_J\ge19683$. Use
$\boldsymbol c^{(8)}$ in \eqref{eq:relative-bound}. Its monotonicity
in $d$, together with \eqref{eq:small-counts}, gives
\[
 \frac{\Delta}{(d-1)N}\ge L_{\boldsymbol c^{(8)}}(243,19683)
 =\frac{11698260132+129548463869\sqrt3}{15108631434}.
\]
On the other hand, \eqref{eq:genus-upper-b} with
$g=826551$ and $b=190269$ gives the upper bound $921685/59049$.
Replacing $\sqrt3$ by $433/250$ in the lower bound, its difference
from the upper bound is
\[
 \frac{6898481753}{419684206500}>0,
\]
a contradiction. Thus $F=R$.
\end{proof}

\subsection{Reduction for \texorpdfstring{$q_0=3$}{q0=3}}
\begin{proposition}\label{prop:q3-intersection}
If $q_0=3$ and $X$ is not the Ree curve, then, for every choice of $P$
and $x$ as in Proposition~\ref{prop:first-nongap},
\begin{equation}\label{eq:q3-intersection}
 F\cap R=k(x),\qquad [FR:F]=729.
\end{equation}
\end{proposition}
\begin{proof}
The hypothesis implies $F\ne R$, so $e_J\le3$ by
Lemma~\ref{lem:small-intersection}. If $e_J=3$, then $d=243$ and
\eqref{eq:intersection-b} gives $b_J\le1053$.
Since $g=3627$, the upper bound for $\Delta/((d-1)729)$ is $4153/729$.
But \eqref{eq:relative-bound} gives
\[
 \frac{\Delta}{(d-1)729}\ge L_{\boldsymbol c^{(8)}}(27,243)
 =\frac{5279310+19984889\sqrt3}{6880302}.
\]
Using $\sqrt3>433/250$, the difference from $4153/729$ is greater than
$174280937/1720075500>0$. Hence $e_J=1$.
\end{proof}

\begin{remark}\label{rem:trace-limit}
For $q_0=3$ and $F\cap R=k(x)$, the trace method of
Section~\ref{sec:compositum} cannot conclude, whatever coefficient
vector is used. Here $d=729$ and $b=2673$, so \eqref{eq:genus-upper-b}
gives the upper bound $4963/729>6.80$ for $\Delta/((d-1)N)$. Let
$\boldsymbol c$ have autocorrelations with $a_0>0$ and $a_i\ge0$ for
$i\ge1$. In \eqref{eq:relative-bound}, the term of index $i\le5$
equals $(a_i/a_0)\,q^{1-i/2}$, and $a_i\le a_0$ by the Cauchy--Schwarz
inequality. For $i\ge6$ the Weil bound gives
$N_i\ge q^i+1-2gq^{i/2}>dq^3+1$, so those terms are nonpositive. Hence
\[
 L_{\boldsymbol c}(27,729)\le\sum_{i=1}^5 27^{1-i/2}<6.44 .
\]
This is why a different argument is needed in Section~\ref{sec:q3}.
\end{remark}

\section{Completion for \texorpdfstring{$q_0=3$}{q0=3}}\label{sec:q3}
The case $q_0=9$ is settled by Proposition~\ref{prop:q9}. It remains to
complete the case $q_0=3$, using the intersection obtained in
Proposition~\ref{prop:q3-intersection}. By Remark~\ref{rem:trace-limit},
genus bounds alone do not suffice here.

Suppose that $X/\F_{27}$ satisfies \eqref{eq:hypotheses} but is not
isomorphic to the Ree curve. Fix $P\in X(k)$ and choose $x$ as in
Proposition~\ref{prop:first-nongap}. We have
\[
 k=\F_{27},\quad N=729,\quad m=1036,\quad g=3627,\quad C=7253,
 \qquad F\cap R=k(x).
\]
The argument first restricts the small non-gaps of $H(P)$.
Lemma~\ref{lem:gap} uses the second-order congruence
$u^{27}\equiv u+(x^{27}-x)u'\pmod{(x^{27}-x)^2}$.
The operator $\T$ in Lemma~\ref{lem:defect} measures the remaining
term after the second Hasse derivative is included. Applied to a
section of pole order $m_2(P)$, it gives the normal form in
Lemma~\ref{lem:normal}. Proposition~\ref{prop:q3} then contradicts
$F\cap R=k(x)$ in each case.

Let $\A$ be the ring of functions regular away from $P$, and put
\[
 D=x^{27}-x,\qquad \V=L(1036P),\qquad
 u_j=D_x^{(j)}u,\qquad u'=u_1.
\]
Hasse derivatives always carry a subscript, as in $D_x^{(j)}$ or
$D_z^{(j)}$; the letter $D$ alone denotes the function $x^{27}-x$.
We call a section $u\in\V$ nonlinear if
$u\notin\Span_k\{1,x\}$.
For $u\ne0$ write $\mu(u)=\max\{0,-v_P(u)\}$, and set $\mu(0)=0$.
Every nonzero element of $\A$ has $v_P(u)\le0$: otherwise it would
be regular everywhere and vanish at $P$. Thus a nonconstant element
of $\A$ has positive pole order in $H(P)$. In particular, for
$u\in\A\setminus\{0\}$ and $j\ge0$,
\begin{equation}\label{eq:product-poles}
 \mu(x^ju)=729j+\mu(u),\qquad
 \mu(Du)=19683+\mu(u).
\end{equation}
Any estimate obtained below by subtracting the pole order of $x^j$
or $D$ is applied to a nonzero product; if the remaining factor is
zero, the asserted upper bound holds automatically.

Since $(dx)=7252P$, differentiation preserves $\A$ and
\begin{equation}\label{eq:ordinary-budget}
 \mu(u')\le\mu(u)+7253.
\end{equation}
Indeed, $du=u'\,dx$ and a pole of order $M$ contributes a pole of
order at most $M+1$ to $du$. If $3\mid M$ and $du\ne0$, the leading
term disappears upon differentiation, so $v_P(du)\ge-M$ instead.

Write $Y,Z\in R$ for the generators $y,z$ of \eqref{eq:ree-model};
capitals are used because $z$ denotes a local parameter below. They satisfy
$Y^{27}-Y=x^3D$ and $Z^{27}-Z=x^6D$. Put $W=x^{10}-Y^9$.
Then
\begin{align}
 W^{27}-W
 &=x^{270}-x^{10}-(Y^{27}-Y)^9\notag\\
 &=x^{270}-x^{10}-x^{27}(x^{243}-x^9)
 =x^{36}-x^{10}=x^9D.\label{eq:W}
\end{align}
Consequently, if $\eta\in F$ satisfies
\begin{equation}\label{eq:Ree-membership}
 \eta^{27}-\eta=D(\kappa_0+\kappa_3x^3+\kappa_6x^6+\kappa_9x^9),\qquad \kappa_i\in k,
\end{equation}
then $\eta\in F\cap R=k(x)$. To prove this, subtract
$\kappa_0x+\kappa_3Y+\kappa_6Z+\kappa_9W$ from $\eta$ and use that all roots of
$T^{27}-T$ belong to $k$. If also $\eta\in\A$, it is a polynomial
in $x$: a finite pole in $k(x)$ would remain a pole at a place of
$F$ different from $P$. Its positive pole order is therefore
divisible by $729$.

For $u\in\V$, Proposition~\ref{prop:section-identities} gives
\begin{equation}\label{eq:small-sections}
 u^{27}-u=Du',\quad u'=h_u^3,\quad
 w_u=u^9-x^9h_u,\quad u=w_u^3+xh_u^3.
\end{equation}
These functions belong to $\A$. If $u\notin\Span_k\{1,x\}$ has
pole order $a$, then
\begin{equation}\label{eq:small-poles}
 \mu(h_u)=9(a-729),\qquad \mu(w_u)=9(a-729)+243.
\end{equation}
The first equality follows from $u^{27}-u=Dh_u^3$. For the second,
the pole order $729+27(a-729)$ of $xh_u^3$ is greater than $a$,
since $a>729$, so it cannot cancel with $u$ in $w_u^3=u-xh_u^3$.
Finally, Lemma~\ref{lem:second-nongap-upper} gives
\begin{equation}\label{eq:m2-upper}
 m_2(P)\le1036-54=982.
\end{equation}

\subsection{Small pole orders and the cubic operator}
\begin{lemma}\label{lem:gap}
One has $H(P)\cap(1036,1458)=\varnothing$. The only positive non-gap
below $1458$ that is not divisible by three is $1036$. Moreover,
\begin{equation}\label{eq:L1458}
 L(1458P)=\V\oplus kx^2.
\end{equation}
\end{lemma}
\begin{proof}
Let $u\in\A$ have positive pole order $a<1458$, and put
$G=u^{27}-u-Du'$. At each of the $19683$ finite rational points
$Q$, the function $z=x-x(Q)$ is a uniformizer and $D=-z+z^{27}$.
Writing $u=c_0+c_1z+O(z^2)$, with $c_i\in k$, gives
$u^{27}\equiv c_0$, $u\equiv c_0+c_1z$ and $Du'\equiv-c_1z$
modulo $z^2$. Thus $v_Q(G)\ge2$.
On the other hand, the only possible pole of $G$ is $P$, of order
at most
\[
 \max\{27a,\,19683+a+7253\}<39366=2\cdot19683.
\]
A nonzero function has zero and pole divisors of the same degree;
hence $G=0$. This also gives $u'\ne0$ and
\[
 27(a-729)=-v_P(u')\le a+7253,
 \qquad 26a\le26936=26\cdot1036.
\]
Thus $a\le1036$. If $3\nmid a$, then $v_P(du)=-a-1$, so the
last inequality is an equality and $a=1036$. This integer is a
non-gap by \eqref{eq:xf-poles}.

The pole order of $x^2$ is $1458$. For any other function of that
order, a scalar multiple of $x^2$ cancels its leading term, because
$P$ is rational. The remainder has pole order less than $1458$,
and therefore at most $1036$. This proves \eqref{eq:L1458}.
\end{proof}

\begin{lemma}\label{lem:defect}
For every $u\in\A$ there is a unique $\T(u)\in\A$ such that
\begin{equation}\label{eq:defect}
 \T(u)^3=\frac{u^{27}-u-Du'-D^2u_2}{D^3}.
\end{equation}
The operator $\T$ is additive. Put $M=\mu(u)$. Then:
\begin{enumerate}
\item if $M<2187$, then $\T(u)=0$; if $M=2187$, then $\T(u)\in k$;
\item if $\T(u)=0$, then $M\le2072$;
\item if $M>2187$, then $\mu(\T(u))=9(M-2187)$;
\item if $1458<M<2187$, then $u_2=r^3$ for some $r\in\A$ with
$\mu(r)=9(M-1458)$.
\end{enumerate}
\end{lemma}
\begin{proof}
At a finite rational point use the same parameter $z$ as above and
write $u=c_0+c_1z+c_2z^2+O(z^3)$. Modulo $z^3$,
\[
 u^{27}=c_0,\quad Du'=-c_1z-2c_2z^2,\quad D^2u_2=c_2z^2.
\]
The constant term in the numerator of \eqref{eq:defect} is zero,
and its coefficients of $z,z^2$ are respectively
\[
 -c_1+c_1=0,\qquad -c_2+2c_2-c_2=0.
\]
These rational points are all the zeros of $D$, each simple.
Both $u'$ and $u_2=2u''$ are in $\A$, so the quotient is regular
away from $P$.

Let $B=u^{27}-u-Du'-D^2u_2$. The composition rule gives
$u''=2u_2$ and $(u_2)'=3u_3=0$. Since $D'=-1$,
\[
 B'=-u'+u'-2Du_2+2Du_2=0,
 \qquad (D^3)'=0.
\]
Thus $B/D^3\in\ker D_x=F^3$. Its unique cube root belongs to $\A$,
because regularity of a cube implies regularity of its root at
every place. Additivity follows by cubing both sides and using
injectivity of cubing in $F$.

By \eqref{eq:ordinary-budget},
\[
 \mu(u')\le M+7253,\qquad \mu(u_2)\le M+14506.
\]
The four numerator terms consequently have pole order at most
\[
 27M,\quad M,\quad M+26936,\quad M+53872,
\]
where $26936=19683+7253=26\cdot1036$ and
$53872=39366+14506=26\cdot2072$.
For $M\le2186$, all four bounds are strictly less than
$59049=\mu(D^3)$; at $M=2187$ they are at most $59049$.
The quotient is therefore regular everywhere, and vanishes at $P$
when $M<2187$. This proves (1).

If $M>2072$, then
\[
 27M>M+53872\ge M+26936\ge M.
\]
The term $u^{27}$ is the unique dominant term in $B$, so $B\ne0$.
This proves (2). When $M>2187$, it also gives
\[
 3\mu(\T(u))=27M-59049=27(M-2187),
\]
which proves (3).

For (4), assertion (1) gives
$D^2u_2=u^{27}-u-Du'$. Since $M>1458>1036$, the term $u^{27}$
strictly dominates the other two. Hence
\[
 \mu(u_2)=27M-39366=27(M-1458).
\]
Now $(u_2)'=0$, and taking a cube root proves the assertion.
\end{proof}

\begin{corollary}\label{cor:small-closure}
If $u\in\A$, $u_2=0$ and $\mu(u)<2187$, then $u\in\V$.
\end{corollary}
\begin{proof}
The defect lemma gives $u^{27}-u=Du'$. For nonconstant $u$, comparison
of poles gives $27\mu(u)\le19683+\mu(u)+7253$, hence
$\mu(u)\le1036$. Constants already belong to $\V$.
\end{proof}

\begin{lemma}\label{lem:coarse}
The second non-gap $t=m_2(P)$ satisfies $891\le t\le981$ and $3\mid t$.
\end{lemma}
\begin{proof}
If $u\in\V\setminus\Span_k\{1,x\}$ has pole order $a$, then
$h_u$ is nonconstant and \eqref{eq:small-poles} gives
$9(a-729)\ge729$. Thus every such $a$ is at least $810$.

Suppose $729<a<891$ is a non-gap, and choose $u$ of pole order $a$.
Put $M=\mu(h_u)=9(a-729)<1458$. Lemma~\ref{lem:gap} gives
$M\le1036$. Hence $\mu(w_u)=M+243\le1279<1458$, and the same
lemma gives $M+243\le1036$. Thus $729\le M\le793$.
If $M>729$, the function $h_u$ is a nonlinear section of $\V$, so
its pole order is at least $810$, a contradiction. Therefore
$M=729$ and $a=810$.

In this case $h_u=\alpha x+\delta$, where $\alpha\in k^\times$
and $\delta\in k$, and
\[
 u^{27}-u=D(\alpha^3x^3+\delta^3).
\]
By \eqref{eq:Ree-membership}, $u\in k(x)$, contrary to its pole order
$810$. This proves $t\ge891$. Equation \eqref{eq:m2-upper} and
Lemma~\ref{lem:gap} give $t\le982$ and $3\mid t$, hence $t\le981$.
\end{proof}

\subsection{Reduction of the second non-gap}
Choose $u\in\A$ with pole order $t=m_2(P)$, and put
\begin{equation}\label{eq:normal-notation}
 u'=h^3,\qquad w=u^9-x^9h,\qquad
 h=a^3+xb^3+x^2r^3,\qquad c=\T(h),\quad S=\T(h').
\end{equation}
The $3$-basis decomposition gives
\begin{equation}\label{eq:abr-cubes}
 r^3=h_2=2h'',\qquad b^3=h'+xr^3,\qquad
 a^3=h-xb^3-x^2r^3.
\end{equation}
Thus $a,b,r\in\A$: all right sides are regular away from $P$,
and so are their cube roots.

\begin{lemma}\label{lem:reduce972}
In \eqref{eq:normal-notation}, $S'=0$, $c\in k$ and $t\le972$.
Moreover,
\begin{equation}\label{eq:normal9}
 h^9=a+x^9b+x^{18}r+cD,
\end{equation}
and
\begin{equation}\label{eq:S}
 DS=(h')^9-b+x^9r.
\end{equation}
\end{lemma}
\begin{proof}
We first record the bounds before $t\le972$ has been established.
By Lemma~\ref{lem:coarse} and \eqref{eq:small-poles},
\[
 \mu(h)\le9(981-729)=2268,\qquad
 \mu(w)\le2268+243=2511.
\]
The pole order of $w$ is divisible by three. If $h'\ne0$, the identity
$dw=-x^9h'\,dx$ gives
\[
 7252-6561-\mu(h')=v_P(dw)\ge-\mu(w),
\]
and therefore $\mu(h')\le2511+7252-6561=3202$. This bound also
holds when $h'=0$.
Equations \eqref{eq:abr-cubes} and \eqref{eq:ordinary-budget} now give
\begin{align}
 \mu(r)&\le\left\lfloor\frac{3202+7253}{3}\right\rfloor=3485,
 \label{eq:rough-r}\\
 \mu(b)&\le\left\lfloor\frac{\max\{3202,729+3\cdot3485\}}3
                    \right\rfloor=3728,\notag\\
 \mu(a)&\le\left\lfloor\frac{
 \max\{2268,729+3\cdot3728,1458+3\cdot3485\}}3
                    \right\rfloor=3971.\notag
\end{align}
Consequently
\begin{equation}\label{eq:rough-abprime}
 \mu(a')\le3971+7253=11224,\qquad
 \mu(b')\le3728+7253=10981.
\end{equation}
The three cases in Lemma~\ref{lem:defect}(1),(3) give
\[
 \mu(c)\le9(2268-2187)=729,\qquad
 \mu(S)\le9(3202-2187)=9135.
\]
Since $L(729P)=\Span_k\{1,x\}$, write $c=\alpha x+\delta$ with
$\alpha,\delta\in k$.

We next derive \eqref{eq:normal9} and \eqref{eq:S}. Substitution of
$h=a^3+xb^3+x^2r^3$ into the definition of $c$ yields
\begin{align*}
 h^{27}
 &=h+Dh'+D^2r^3+c^3D^3\\
 &=a^3+(x+D)b^3+(x^2-xD+D^2)r^3+c^3D^3\\
 &=a^3+x^{27}b^3+x^{54}r^3+c^3D^3.
\end{align*}
Here $x+D=x^{27}$ and $x^2-xD+D^2=x^{54}$ in characteristic
three. Taking a cube root gives \eqref{eq:normal9}.
Likewise $h'=b^3-xr^3$, $h''=-r^3$ and $(h')_2=0$, so
\[
 D^3S^3=(h')^{27}-h'-Dh''=(h')^{27}-b^3+x^{27}r^3.
\]
Taking a cube root gives \eqref{eq:S}.

Differentiating these two identities gives
\begin{equation}\label{eq:roughder}
 b'=x^9r'+S-DS',\qquad
 a'+x^9b'+x^{18}r'=c-\alpha D.
\end{equation}
The second identity, together with \eqref{eq:rough-abprime}, implies
\[
 \mu(x^{18}r')\le
 \max\{11224,6561+10981,729,19683\}=19683.
\]
Hence $\mu(r')\le19683-13122=6561$. Returning to the first identity,
\[
 \mu(DS')\le\max\{6561+6561,9135,10981\}=13122<19683.
\]
If $S'\ne0$, \eqref{eq:product-poles} would give
$\mu(DS')=19683+\mu(S')\ge19683$. Thus $S'=0$.

Now $x^9r'=b'-S$, and therefore
\[
 \mu(r')\le\max\{10981,9135\}-6561=4420.
\]
It follows that the left side of the second identity in
\eqref{eq:roughder} has pole order at most
\[
 \max\{11224,6561+10981,13122+4420\}=17542.
\]
If $\alpha\ne0$, its right side $c-\alpha D$ has pole order
$19683$, a contradiction. Hence $\alpha=0$ and $c\in k$.
Finally, if $\mu(h)>2187$, Lemma~\ref{lem:defect}(3) would make
$c$ nonconstant. Thus $9(t-729)=\mu(h)\le2187$, proving $t\le972$.
\end{proof}

\subsection{The third Hasse derivative}
\begin{lemma}\label{lem:two-small-functions}
Put $v=r'$ and $\xi=S-x^9v$. Then $v,\xi\in\V$,
\begin{equation}\label{eq:xiprime}
 \xi'=-x^9v',
\end{equation}
and $v,\xi$ are constant.
\end{lemma}
\begin{proof}
We first update the estimates using $t\le972$:
\begin{align}
 \mu(h)&\le9(972-729)=2187,\qquad \mu(w)\le2187+243=2430,
 \notag\\
 \mu(h')&\le2430+7252-6561=3121,\notag\\
 \mu(r)&\le\left\lfloor\frac{3121+7253}{3}\right\rfloor=3458,
 \label{eq:sharp-poles}\\
 \mu(b)&\le\frac{\max\{3121,729+3\cdot3458\}}3=3701,\notag\\
 \mu(a)&\le\frac{\max\{2187,729+3\cdot3701,1458+3\cdot3458\}}3
 =3944.\notag
\end{align}
In particular,
\[
 \mu(a')\le3944+7253=11197,\qquad
 \mu(S)\le9(3121-2187)=8406.
\]
Since $S'=0$ and $c\in k$, \eqref{eq:roughder} becomes
\begin{equation}\label{eq:sharpder}
 b'=x^9v+S,\qquad a'=c+x^{18}v-x^9S.
\end{equation}
The second identity gives
\[
 \mu(v)\le\max\{11197,6561+8406\}-13122=1845.
\]
Moreover $v_2=D_x^{(2)}D_xr=3D_x^{(3)}r=0$.
Since $1845<2187$, Corollary~\ref{cor:small-closure} gives
\begin{equation}\label{eq:v-sharp}
 v\in\V,\qquad \mu(v)\le1036.
\end{equation}

We shall use the third Hasse derivative $h_3$. Applying $D_x^{(3)}$
to the $3$-basis decomposition of $h$ gives
\[
 h_3=(a')^3+x(b')^3+x^2v^3.
\]
Cubing $a'+x^9b'+x^{18}v=c$, and subtracting, yields
\[
 h_3=c^3-D(b')^3+(Dx-D^2)v^3,
\]
since $x^2-x^{54}=Dx-D^2$. Substitution of $b'=x^9v+S$ gives
\begin{align}
 h_3
 &=c^3-DS^3+(-Dx^{27}+Dx-D^2)v^3\notag\\
 &=c^3-DS^3+D^2v^3.\label{eq:h3}
\end{align}
In the last step, $-Dx^{27}+Dx=-D^2$ and $-2=1$ in $k$.

To bound $h_3$, choose a uniformizer $z$ at $P$ and write
$x=\sum_{n\ge-729}a_nz^n$. Since $v_P(D_zx)=7252$, every exponent
$n<7253$ with $a_n\ne0$ is divisible by three. For such integers,
including negative integers, $\binom n2\equiv0\pmod3$.
It follows that
\[
 v_P(D_zx)=7252,\quad v_P(D_z^{(2)}x)\ge7251,\quad
 v_P(D_z^{(3)}x)\ge-732.
\]
Also $v_P(D_z^{(3)}h)\ge-2187-3=-2190$.
The change-of-parameter formula \eqref{eq:chain-three} gives
\[
 h_3(D_zx)^3=D_z^{(3)}h-h'D_z^{(3)}x
             -2r^3(D_zx)(D_z^{(2)}x).
\]
The three numerator terms have respective lower valuation bounds
\begin{align*}
 v_P(D_z^{(3)}h)&\ge-2190,\\
 v_P(h'D_z^{(3)}x)&\ge-3121-732=-3853,\\
 v_P(2r^3(D_zx)(D_z^{(2)}x))
   &\ge-3\cdot3458+7252+7251=4129.
\end{align*}
Since $v_P((D_zx)^3)=3\cdot7252=21756$, we obtain
\begin{equation}\label{eq:h3-bound}
 \mu(h_3)\le21756-\min\{-2190,-3853,4129\}=25609.
\end{equation}

For $\xi=S-x^9v$, equation \eqref{eq:h3} reads
\[
 h_3=c^3-D(\xi^3+xv^3),\qquad
 D\xi^3=c^3-h_3-Dxv^3.
\]
Using \eqref{eq:v-sharp} and \eqref{eq:h3-bound}, we get
\begin{align*}
 3\mu(\xi)
 &\le\max\{25609,19683+729+3\cdot1036\}-19683\\
 &=\max\{5926,3837\}=5926.
\end{align*}
Hence $\mu(\xi)\le1975$. Because $S'=0$, the function $S$ is a
cube and $S_2=0$. Also $(x^9v)_2=x^9v_2=0$. Thus $\xi_2=0$, and
Corollary~\ref{cor:small-closure} gives $\xi\in\V$.
Differentiation yields \eqref{eq:xiprime}.

Suppose $v\notin\Span_k\{1,x\}$ has pole order $\tau$.
Then $\tau\ge t\ge891$ and the Frobenius identity gives
\[
 \mu(v')=27(\tau-729),\qquad
 \mu(\xi')=6561+27(\tau-729).
\]
The latter number is positive. Applying the same identity to $\xi$
therefore gives $27(\mu(\xi)-729)=6561+27(\tau-729)$, whence
\[
 \mu(\xi)=\tau+243\ge1134>1036,
\]
a contradiction. Thus $v=\alpha x+\beta$ with $\alpha,\beta\in k$.
If $\alpha\ne0$, then
\[
 \xi^{27}-\xi=-\alpha x^9D,
 \qquad \mu(\xi)=\frac{19683+6561}{27}=972.
\]
By \eqref{eq:Ree-membership}, $\xi\in k(x)$, contradicting this pole
order. Hence $v=\beta\in k$. Now \eqref{eq:xiprime} gives $\xi'=0$,
and $\xi^{27}-\xi=D\xi'=0$ implies $\xi\in k$.
\end{proof}

For reference, the bounds used in the normal form are collected below.
The first column applies before Lemma~\ref{lem:reduce972}; the second
uses its conclusion $t\le972$. Each entry is an upper bound for $\mu$.
\[
\begin{array}{c|r|r|l}
 \text{function}&t\le981&t\le972&\text{source}\\\hline
 h&2268&2187&\eqref{eq:small-poles}\\
 w&2511&2430&\eqref{eq:small-poles}\\
 h'&3202&3121&dw=-x^9h'\,dx\\
 r&3485&3458&\eqref{eq:abr-cubes},\ \eqref{eq:ordinary-budget}\\
 b&3728&3701&\eqref{eq:abr-cubes}\\
 a&3971&3944&\eqref{eq:abr-cubes}\\
 S&9135&8406&\text{Lemma~\ref{lem:defect}}\\
 h_3&\text{--}&25609&\eqref{eq:h3-bound}
\end{array}
\]

\subsection{The differential normal form}
\begin{lemma}\label{lem:normal}
There are $\beta,\varepsilon\in k$ and $\rho,\sigma,\theta\in\A$ such that
\begin{align}
 r&=\beta x+\rho^3,&
 b&=\varepsilon x-\beta x^{10}+\sigma^3,
 \label{eq:rb}\\
 \theta^3&=\sigma-x^3\rho,&
 \theta^{27}-\theta&=D(\beta^3x^3-\rho^9).
 \label{eq:theta}
\end{align}
Their pole orders satisfy
\begin{equation}\label{eq:rho-sigma-bounds}
 \mu(\rho)\le1152,\qquad \mu(\sigma)\le2430.
\end{equation}
\end{lemma}
\begin{proof}
By Lemma~\ref{lem:two-small-functions}, write $v=\beta$ and
$\xi=\varepsilon$, with $\beta,\varepsilon\in k$.
Equations \eqref{eq:sharpder} and $\xi=S-x^9v$ give
\[
 r'=\beta,\qquad S=\varepsilon+\beta x^9,
 \qquad b'=\varepsilon-\beta x^9.
\]
Thus $r-\beta x$ and $b-\varepsilon x+\beta x^{10}$ have derivative
zero. Their cube roots $\rho,\sigma$ belong to $\A$, proving
\eqref{eq:rb}. The updated bounds \eqref{eq:sharp-poles} give
\[
 \mu(\rho)\le\left\lfloor\frac{\max\{3458,729\}}3\right\rfloor=1152,
 \qquad
 \mu(\sigma)\le\frac{\max\{3701,729,7290\}}3=2430.
\]

Since $h'=b^3-xr^3$, equation \eqref{eq:S} is equivalent to
\[
 b^{27}-b=x^9(r^{27}-r)+DS.
\]
Substitute \eqref{eq:rb} and $S=\varepsilon+\beta x^9$. The two
sides become respectively
\[
 \varepsilon D-\beta(x^{270}-x^{10})+(\sigma^{27}-\sigma)^3
\]
and
\[
 \varepsilon D-\beta x^9D+x^9(\rho^{27}-\rho)^3.
\]
Consequently
\[
 (\sigma^{27}-\sigma)^3
 =x^9(\rho^{27}-\rho)^3+\beta(x^{270}-x^{36}).
\]
In $k=\F_{27}$, the cube root of $\beta$ is $\beta^9$; hence
\begin{equation}\label{eq:sigma-frob}
 \sigma^{27}-\sigma
 =x^3(\rho^{27}-\rho)+\beta^9(x^{90}-x^{12}).
\end{equation}
Differentiation gives $\sigma'=x^3\rho'$, so
$\sigma-x^3\rho=\theta^3$ for a unique $\theta\in\A$.
Finally,
\begin{align*}
 (\theta^{27}-\theta)^3
 &=\sigma^{27}-\sigma-x^{81}\rho^{27}+x^3\rho\\
 &=-(x^{81}-x^3)\rho^{27}+\beta^9(x^{90}-x^{12})\\
 &=-D^3\rho^{27}+\beta^9(x^3D)^3\\
 &=\bigl(D(\beta^3x^3-\rho^9)\bigr)^3.
\end{align*}
Injectivity of cubing proves the second identity in \eqref{eq:theta}.
\end{proof}

\subsection{Completion of the proof}
\begin{proposition}\label{prop:q3}
Theorem~\ref{thm:main} holds for $q_0=3$.
\end{proposition}
\begin{proof}
Under the counterexample hypothesis, the preceding lemmas give
$891\le t=m_2(P)\le972$ and the normal form \eqref{eq:rb}--\eqref{eq:theta}.

Suppose first that $\rho$ is nonconstant, and write $p=\mu(\rho)$.
Then $729\le p\le1152$. Since
\[
 \mu(x^3\rho)=2187+p>2430\ge\mu(\sigma),
\]
there is no cancellation in $\theta^3=\sigma-x^3\rho$, and
\begin{equation}\label{eq:theta-pole}
 \mu(\theta)=729+p/3\in[972,1113].
\end{equation}
The gap interval of Lemma~\ref{lem:gap} forces $\mu(\theta)\le1036$;
hence $p\le921$.

If $p>729$, the definition of $t$ gives $t\le p\le921$.
Thus $\mu(h)=9(t-729)\le1728<2187$, so $\T(h)=c=0$.
Also $\mu(r)=3p$, since $r=\beta x+\rho^3$ and $3p>729$.
If $\mu(h)\le1458$, equation \eqref{eq:L1458} would give
$h\in\V+kx^2$, so $h_2=r^3$ would be constant, a contradiction.
Therefore $1458<\mu(h)<2187$, and Lemma~\ref{lem:defect}(4) gives
\[
 3p=\mu(r)=9(\mu(h)-1458)=81(t-891).
\]
It follows that $t=891+p/27$, and consequently
\[
 t-p=891-\frac{26p}{27}
 \ge891-\frac{26\cdot921}{27}=\frac{37}{9}>0,
\]
contrary to $t\le p$.

It remains in this case that $p=729$. Then
$\rho=\alpha x+\delta$ with $\alpha\in k^\times$, and
\eqref{eq:theta-pole} and \eqref{eq:theta} give
\[
 \mu(\theta)=972,\qquad
 \theta^{27}-\theta=D(\beta^3x^3-\alpha^9x^9-\delta^9).
\]
Equation \eqref{eq:Ree-membership} forces $\theta\in k(x)$, again
contradicting its pole order. Thus $\rho$ must be constant.

If $\beta\ne0$, the identity
$\sigma^3=b-\varepsilon x+\beta x^{10}$ has a unique dominant term
$\beta x^{10}$, because $\mu(b)\le3701<7290$.
It gives $\mu(\sigma)=2430$. Since $\rho$ is constant,
$\mu(x^3\rho)\le2187$, and $\theta^3=\sigma-x^3\rho$ gives
$\mu(\theta)=810$. This lies strictly between $m_1(P)=729$ and
$m_2(P)\ge891$, a contradiction. Therefore $\beta=0$ and $r$ is constant.

Put $h_0=h-c^3x^3$. For $c\in k$,
\[
 \T(c^3x^3)^3
 =\frac{c^{81}x^{81}-c^3x^3}{D^3}=c^3,
\]
because $c^{81}=c^3$ and $D^3=x^{81}-x^3$.
Thus $\T(c^3x^3)=c$ and $\T(h_0)=0$.
Lemma~\ref{lem:defect}(2) gives $\mu(h_0)\le2072$.
If $\mu(h_0)>1458$, its assertion (4), applied to
$(h_0)_2=h_2=r^3$, would make $r$ nonconstant. Hence
$\mu(h_0)\le1458$, and \eqref{eq:L1458} gives
\[
 h_0=\lambda x^2+v_0,\qquad \lambda\in k,\quad v_0\in\V.
\]
If $v_0$ is nonlinear, its pole order is at least $t\ge891$,
and \eqref{eq:small-sections} gives
\[
 \mu(v_0')\ge27(891-729)=4374.
\]
This pole cannot cancel with $-\lambda x$ in $h'=-\lambda x+v_0'$,
contrary to $\mu(h')\le3121$. Therefore
$v_0\in\Span_k\{1,x\}$ and $h\in k[x]$ has degree at most three.

Write $h=d_0+d_1x+d_2x^2+d_3x^3$ with $d_i\in k$. Then
\[
 u^{27}-u=Dh^3=D(d_0^3+d_1^3x^3+d_2^3x^6+d_3^3x^9).
\]
Equation \eqref{eq:Ree-membership} gives $u\in F\cap R=k(x)$,
contradicting $891\le\mu(u)=t\le972$. This completes the proof.
\end{proof}

\begin{proof}[Proof of Theorem~\ref{thm:main}]
Proposition~\ref{prop:lower-five} proves the theorem for $q_0\ge27$,
Proposition~\ref{prop:q9} proves it for $q_0=9$, and
Proposition~\ref{prop:q3} proves it for $q_0=3$.
\end{proof}

\section*{Acknowledgements}
The author was partially supported by CNPq, grant no.~302774/2025-4,
and FAPESP, grant no.~2024/00923-6.

\section*{Declaration on the use of generative AI}
ChatGPT (OpenAI, GPT-6) was used for language editing, mathematical
exploration, and assistance with the arguments for the small parameters
$q_0=3$ and $q_0=9$, including symbolic identities, pole-order estimates,
and the preparation of verification code. Claude (Anthropic, Claude Opus
5.5) was used to review the manuscript, to recheck its numerical
certificates, and for editing of the final text. The verification code is
available from the author upon request. The author is responsible for the
mathematical arguments, references, and final text.

\end{document}